\documentclass[11pt]{amsart}
\usepackage[margin=1.15in]{geometry}
\usepackage{amsthm}
\usepackage{amssymb}
\usepackage[shortlabels]{enumitem}
\usepackage{mathtools}
\usepackage{xcolor}
\usepackage{hyperref}
\usepackage{verbatim}

\newcommand{\scal}{\mathrm{R}}
\newcommand{\hess}{\mathrm{Hess}}
\newcommand{\R}{\mathbb{R}}
\newcommand{\C}{\mathcal{C}}

\renewcommand{\r}{\rightarrow}
\newcommand{\ric}{\mathrm{Ric}}
\renewcommand{\div}{\mathrm{div}}
\newcommand{\ov}{\overline}
\renewcommand{\geq}{\geqslant}
\renewcommand{\leq}{\leqslant}
\newcommand{\tr}{\mathrm{tr}}

\DeclareMathOperator{\Rm}{\operatorname{Rm}}

\newtheorem{theorem}{Theorem}[subsection]

\newtheorem{corollary}[theorem]{Corollary}
\newtheorem{lemma}[theorem]{Lemma}

\newtheorem{proposition}[theorem]{Proposition}

\newtheorem*{remark}{Remark}

\theoremstyle{definition}
\newtheorem{definition}{Definition}[section]

\title[Dynamical functionals on ancient ARF Ricci flows]{Dynamical functionals on ancient ARF Ricci flows}
\author[Isaac M. Lopez]{Isaac M. Lopez}
\author[Rio Schillmoeller]{Rio Schillmoeller}
\address{University of Chicago, Department of Mathematics, 5734 S University Avenue, Chicago, IL 60637, USA.}
\email{iml@uchicago.edu}
\address{Stanford Online High School, 415 Broadway Academy Hall, Floor 2, 8853, 415 Broadway, Redwood City, CA 94063.}
\email{rioschillmoeller@gmail.com}

\numberwithin{equation}{section}
\mathtoolsset{showonlyrefs}

\begin{document}

\begin{abstract}
    We introduce a dynamical energy functional on compact ancient asymptotically Ricci-flat Ricci flows with modest decay using limits of conjugate heat flows. This functional satisfies a steady Ricci breather-type rigidity and provides an upper bound for the ordinary $\lambda$-functional while retaining many of its properties. In addition, motivated by work of Colding and Minicozzi, we derive local eigenvalue estimates for normalized Ricci flows coupled with conjugate heat flows.
\end{abstract}

\maketitle

\setcounter{tocdepth}{1}

\tableofcontents

\section{Introduction}

Perelman's functionals have been extensively studied on solutions to the Ricci flow equation defined on compact manifolds. A classical result in this setting asserts that if the $\lambda$-functional assumes the same value at two distinct times along a Ricci flow (that is, if the flow is a steady Ricci breather), then the solution in question is a Ricci-flat, steady gradient Ricci soliton between those two times. These functionals have found a myriad of other applications to geometric problems concerning compact $3$-manifolds. Their utility has recently catalyzed efforts to relate them to the limiting behavior of the underlying Ricci flow solutions on which they are defined. One such application is the dynamical stability of Ricci-flat metrics. A metric is said to be \emph{dynamically stable} if all sufficiently close metrics converge to it along the Ricci flow. In \cite{Ye1993}, it is shown that any sufficiently pinched Einstein metric of nonzero scalar curvature is dynamically stable in the $C^2$-topology. \cite{GuentherIsenbergKnopf2002} addresses the same question for Ricci-flat metrics, concluding that any metric in a sufficiently small $C^{2, \alpha}$-neighborhood of a Ricci-flat metric converges exponentially quickly to a Ricci-flat metric. Furthermore, \cite{Haslhofer-Muller} and \cite{sesum} connect dynamical stability to \emph{linear stability}, which is related to the stability of the Perelman's energy functional at that metric, understood in the classical sense.

Another research direction related to geometric functionals on compact manifolds was initiated by Hein and Naber. In \cite{hein-naber}, they introduce a time-dependent, pointed entropy functional along a solution to the Ricci flow that is defined in terms of Perelman's entropy functional evaluated at a pointed conjugate heat kernel. Lower bounds for the pointed entropy may be converted into local curvature upper bounds along the Ricci flow in question. Analogues of the pointed entropy have found applications to problems concerning classifications of Ricci flows on certain classes of non-compact manifolds, as discussed in \cite{bamler-21-2}, \cite{bamler-21-1}, \cite{chan-ma-zhang-1}, \cite{chan-ma-zhang-2}, and \cite{lopez-ozuch}.

The principal purpose of this paper is to liaise between the dynamical stability of Ricci-flat metrics and dynamical functionals defined on ancient Ricci flows on compact manifolds converging to Ricci-flat metrics. More precisely, we define a canonical dynamical energy functional on certain classes of such flows in terms of conjugate heat flows. By canonical, we intend to emphasize that the functional is independent of the time at which we begin running the heat flow backwards in time. The range of ancient flows on which this functional is defined is dictated by the Lojasiewicz-Simons inequality of Haslhofer-Müller \cite[Theorem 1.3]{Haslhofer-Muller}, which plays a salient role in proving their dynamical stability results.

\subsection{Main results}
On compact Ricci flows with bounded curvature tending to a Ricci-flat metric at a fast enough rate (as formalized in the proceeding result), we derive a monotone dynamical $\lambda$-functional denoted $\lambda_{\mathrm{dyn}}^\infty$, akin to the pointed entropy functional introduced in \cite{hein-naber}, that serves as an upper bound for Perelman's $\lambda$-functional. This new functional further provides an alternative method of classifying Ricci breathers along such Ricci flows.

\begin{theorem} \label{main-thm-arf-cpct}
Suppose that $(M^n,g(t))$ is an ancient solution to the Ricci flow equation on a compact manifold for which the Lojasiewicz-Simon inequality for the $\lambda$-functional,
\begin{align} \label{eq:main-thm-1}
    ||\ric_{f_{g(t)}}||_{L^2(e^{-f_g(t)}dV_g(t))} \geq |\lambda[g(t)]|^{1-\theta},
\end{align}
holds for some $\theta \in [\frac{2}{5},\frac{1}{2}]$ for all $t\in (-\infty,T_0]$ for some $T_0<0$, where $f_g$ is the minimizer of $\lambda$. Suppose further that there exist constants $C_k,T_k>0$ independent of $t$ such that the inequalities
\begin{align} \label{eq:main-thm-2}
    |\nabla^k[g(t)-g_{\mathrm{RF}}]| \leq C_k|t|^{-\frac{\theta}{1-2\theta} }
\end{align}
hold for $k\in \{1,2,3\}$ for all $t\in (-\infty,T_k]$.

\begin{enumerate}[(a)]
    \item There exists a solution $f^\infty:M\times [0,\infty) \r \R$ to the conjugate heat flow such that the functional
    \begin{align}
    \lambda^\infty_{\mathrm{dyn}}(t):=\mathcal{F}[g(t),f^\infty(-t)],
    \end{align}
    which is defined on $(-\infty,0]$, is bounded from below by Perelman's $\lambda$-functional evaluated at $g(t)$.
    \item If there exists a pair of non-positive times $(t_1,t_2)$ such that $t_1<t_2$ and
    \begin{align}
        \lambda^\infty_{\mathrm{dyn}}(t_1) = \lambda^\infty_{\mathrm{dyn}}(t_2),
    \end{align}
    then $(g(t))_{t\in [t_1,t_2]}$ is a Ricci-flat, steady gradient Ricci soliton.
\end{enumerate}
\end{theorem}
In \cite{Haslhofer-Muller}, it is proven that the inequality  \eqref{eq:main-thm-1} holds true with respect to a uniform exponent $\theta \in (0,\frac{1}{2}]$ on a sufficiently small $\C^{2,\alpha}$-neighborhood of any Ricci-flat metric In particular, any ancient solution to the Ricci flow converging to a Ricci-flat metric $g_{\mathrm{RF}}$ satisfies \eqref{eq:main-thm-1} on an arbitrarily large interval of the form $(-\infty,T]$ for some $T<0$. According to \cite{Haslhofer-perelman-lambda}, flows tending to integrable metrics (i.e. Ricci-flat metrics whose infinitesimal Ricci-flat deformations are integrable) such as Calabi-Yau and Kähler metrics satisfy \eqref{eq:main-thm-1} with the optimal exponent $\theta = \frac{1}{2}$. In fact, every known example of a compact ancient Ricci flow satisfies \eqref{eq:main-thm-1} and \eqref{eq:main-thm-2} with respect to the optimal exponent $\theta = \frac{1}{2}$, so our assumption that $\theta \geq \frac{2}{5}$ offers some modest leeway.

Aside from our main result, we conclude the paper by establishing local upper bounds for eigenvalues of the drift Laplacian along a Ricci flow coupled with a conjugate heat flow as well as its sharpness for the first eigenvalue. Such a result is motivated by recent work of Colding and Minicozzi, who derived sharp upper bounds for eigenvalues of the drift Laplacian along a modified Ricci flow \cite{colding-minicozzi}.

\begin{theorem}
 If $(g(t),f(t))$ is a coupled system of Riemannian metrics and smooth functions on a compact manifold $M^n $ satisfying
\begin{align*}
    \begin{cases}
        \partial_tg = -2\ric - 2\nabla^2f \\ \partial_tf = - \scal - \Delta f,
    \end{cases}
\end{align*}
 and $\lambda_k(t)$ denotes the $k$th eigenvalue of the drift Laplacian for $f$ on $(M^n,g(t), f(t))$ for some $k\geq 1$, then we have the local upper bound
    \begin{align*}
        \lambda_k(t) \leq \frac{\lambda_k(t_0)}{1-2(t-t_0)\lambda_k(t_0)} \quad \text{for} \text{ }\text{all} \text{ }  t \in \left[t_0,t_0+\frac{1}{2\lambda_k(t_0)}\right)
        \end{align*}
        for any $t_0$ in the interval of definition of the coupled system.
\end{theorem}

\subsection{Organization of the article}
In Section \ref{curvature-estimates}, 
we establish some curvature estimates for ancient Ricci flows on compact manifolds satisfying \eqref{eq:main-thm-1} and \eqref{eq:main-thm-2} for some $\theta \in [\frac{2}{5},\frac{1}{2}]$ (which we shall refer to as asymptotically Ricci-flat, or ARF for short). In Section \ref{dyn-compact-ARF}, we use these estimates to define a dynamical energy functional for ARF flows. In Section \ref{eigenvalue-bounds}, by adapting the methods of \cite{colding-minicozzi} to the setting of Ricci flows coupled to conjugate heat flows, we prove a local upper bound for eigenvalues of the drift Laplacian for such a coupled system and show that it is sharp for the first eigenvalue. We sequester some of the lengthier calculations used throughout the paper in the appendix.

\subsection{Notation and conventions} We liberally use the following notational conventions throughout the paper.
\begin{itemize}
    \item $dV_g$ denotes the volume form computed with respect to a given Riemannian metric $g$.
    \item $g_e$ denotes the standard Euclidean metric, and $\nabla^e=\nabla^{g_e}$, $dV_e=dV_{g_e}$, etc.
    \item Given a function $f(x)$ and a non-negative function $g(x)$, we write $f=O(g)$ to mean that $|f(x)|\leq cg$ for some constant $c>0$ independent of $x$ for sufficiently large $x$. The variable in consideration will be clear from the context, but it will most often be time.
    \item Given two tensors $A$ and $B$, we denote by $A*B$ a tensor derived from the tensor product $A\otimes B$ in the sense of \cite[Section 2.1]{topping-ricci-flow}. In particular, $|A*B|=O(|A|\cdot |B|)$.
    \item Given a function $f\in \C^\infty(M)$, we denote by $\ric_f$ and $\Delta_f$ the weighted Ricci curvature tensor $\ric_f = \ric + \nabla^2f$ and the drift Laplacian $\Delta_f = \Delta - \nabla_{\nabla f}$.
\end{itemize}
We refer the reader to \cite[Chapter 2]{chow-ricci-flow-1} for all of the relevant evolution equations for geometric quantities along the Ricci flow, many of which we use freely throughout the paper.

\subsection{Acknowledgements} This paper is the outcome of an MIT PRIMES project, where R.S. was mentored by I.M.L. R.S. thanks the PRIMES program for making this opportunity possible. The authors also thank Tristan Ozuch for suggesting the project, and for continuously giving invaluable insights and advice.

\section{Curvature estimates for ARF Ricci flows} \label{curvature-estimates}

\subsection{Definitions} The following definition, which should be compared with the estimates in the hypothesis of Theorem \ref{main-thm-arf-cpct}, makes precise the class of Ricci flows we shall be interested in throughout the paper.

\begin{definition}[Asymptotically Ricci-flat flow]
    Let $(M^n,g(t))_{t\in (-\infty, S]}$ be an ancient Ricci flow on a compact Riemannian manifold $M^n$. We say that $g(t)$ is \emph{asymptotically Ricci-flat (ARF) of order $\theta \in (0,\frac{1}{2}]$ and regularity $m \in \mathbb{N}_0$} if there exists a Ricci-flat metric $g_{\mathrm{RF}}$ and constants $C_k,T_k>0$ for each $k \in \{1,...,m\}$ such that
    \begin{enumerate}[(a)]
        \item the Lojasiewicz-Simon inequality for the $\lambda$-functional,
    \begin{align} \label{eq:lojasiewicz-simon-lambda}
    ||\ric_{f_{g(t)}}||_{L^2(e^{-f_g(t)}dV_g(t))} \geq |\lambda[g(t)]|^{1-\theta},
    \end{align}
    holds for some $\theta \in [\frac{2}{5},\frac{1}{2}]$ for all $t\leq -T_0$, where $f_{g(t)}$ is the minimizer of $\lambda[g(t)]$;
    \item the inequality
    \begin{align} \label{eq:arf-derivative-decay-rates}
    |\nabla^{k,g(t)}[g(t)-g_{\mathrm{RF}}]|_{g(t)} \leq C_k|t|^{-\frac{\theta}{1-2\theta} }
    \end{align}
    holds for all $t\leq -T_k$.
    \end{enumerate}
    As in \cite{Haslhofer-perelman-lambda}, one can interpolate between higher order derivative estimates with respect to $g_{\mathrm{RF}}$ and $g(t)$.
\end{definition}

We shall often find it helpful to write the inequalities from the previous definition in terms of $\beta:=\frac{\theta}{1-2\theta}$. In particular, \eqref{eq:lojasiewicz-simon-lambda} implies that there exist constants $C_0,T_0>0$ such that
\begin{align}
    |g(t)-g_{\mathrm{RF}}| \leq C_0|t|^{-\beta}
\end{align}
for all $t \in (-\infty,0]$ satisfying $t \leq -T_0$. Consequently, if $g(t)$ is ARF of order $\theta$, then the $k$-th order variant of \eqref{eq:arf-derivative-decay-rates} holds true with the exponent $-\beta-k$ for each $k\in \{0,...,m\}$. In particular, $\beta \geq c$ if and only if $\theta \geq \frac{c}{2c+1}$, and the metric decays exponentially quickly to $g_{\mathrm{RF}}$ when $\beta=\infty$ or $\theta= \frac{1}{2}$. For simplicity of notation, we utilize the definition in terms of $\beta$ pervasively throughout the remainder of the paper and write $g(t) \in \mathcal{M}^m_\beta(g_{\mathrm{RF}})$ to mean that $g(t)$ is ARF of order $\frac{\theta}{1-2\theta}$ and regularity $m$.

\subsection{Ricci curvature estimates} We devote the remainder of this section to deriving Ricci curvature estimates for ARF flows that decay polynomially in time. We remark that if $g(t) \in \mathcal{M}^{m+2}_{\beta}(g_{\mathrm{RF}})$, then $|\nabla^k\mathrm{Rm}[g(t)]|_{g(t)}$ is uniformly bounded in $t$ for any $k\leq m$. In fact, if $(M, g_{\mathrm{RF}})$ is Riemann-flat, then $\lvert \mathrm{Rm}[g(t)]\rvert  = O(|t|^{-\beta})$; however, such precise estimates for the norm of the Riemann curvature tensor and its derivatives are frivolous for our purposes, so we will not explore them further.

We begin by establishing an asymptotic estimate for the norm of the Ricci curvature tensor. As we shall make evident momentarily, the proof of this estimate remains valid if we consider immortal flows in lieu of ancient ones.

\begin{lemma} \label{arf-ricci-decay}
    If $g(t) \in \mathcal{M}^2_\beta(g_{\mathrm{RF}})$, then $|\ric[g(t)]|_{g(t)}=O(|t|^{-\beta})$.
    \begin{proof}
        Given any $t\leq -T:=-\max_{k=0,1,2}(T_k)$, we define the linear function $\psi:[0,1] \to S^2 T^*M$ by
        \begin{align}
           \psi(s)=(1-s)g(t)+sg_{\mathrm{RF}}
        \end{align}
        so that $\psi(0)=g(t)$ and $\psi(1)=g_{\mathrm{RF}}$.
        Since $|g(t)-g_{\mathrm{RF}}|=O(|t|^{-\beta})$ and $\delta_{ij}$ is uniformly equivalent to $g_{\mathrm{RF}}$ by virtue of the compactness of $M$, there exist positive constants $c_1,c_2$ such that
        \begin{align}
            (1-c_1c_2^{-1})g_{\mathrm{RF}} \leq g_{\mathrm{RF}}- c_1\delta_{ij}\leq g(t) \leq g_{\mathrm{RF}} + c_1\delta_{ij} \leq (1+c_1c_2)g_{\mathrm{RF}}
        \end{align}
        for all $t\leq -T$. Setting $C=\max(|1-c_1c_2^{-1}|^{-1},|1+c_1c_2|)$, we obtain that $C^{-1}g_{\mathrm{RF}} \leq g(t) \leq Cg_{\mathrm{RF}}$, whence the metrics $(g(t))_{t\leq -T}$ are uniformly equivalent to $g_{\mathrm{RF}}$.  Consequently, we may invoke \cite[Lemma C.1]{deruelle-ozuch-2020} to obtain the upper bound
        \begin{align}
            |\partial_s\ric[\psi(s)]|_{{g(t)}} &\leq A[ (1+|h|_{g(t)})|\nabla^{g(t),2}h|_{g(t)} + |\mathrm{Rm}[g(t)]|_{g(t)}|h|_{g(t)} \\ & \text{ } \text{ } \text{ } \text{ } \text{ } \text{ } \text{ } \text{ }   + (1+|h|_{g(t)})|\nabla^{g(t)}h|^2_{g(t)}] \label{eq:do-linear-variation-ric}
        \end{align}
        for some constant $A$ independent of $s\in [0,1]$ and $t\in (-\infty,-T]$. Since $g(t)$ is assumed to be ARF of order $\beta$ with regularity $2$, we have that $|\nabla^{g(t),k}h|_{g(t)}=O(|t|^{-\beta})$ for $k\in \{0,1,2\}$, so \eqref{eq:do-linear-variation-ric} tells us that $|\partial_s\ric[\psi(s)]|_{g(t)}=O(|t|^{-\beta})$. Consequently, the fundamental theorem of calculus yields that 
        \begin{align}
             |\ric[g(t)]-\ric(g_{\mathrm{RF}})|_{g(t)} \leq  \int^1_0 |\partial_s\ric[\psi(s)]|_{g(t)}ds = O(|t|^{-\beta}).
        \end{align}
        Since $\ric(g_{\mathrm{RF}})=0$, it follows that $|\ric[g(t)]|_{g(t)}=O(|t|^{-\beta})$, as claimed. 
    \end{proof}
\end{lemma}

We continue by using the maximum principle to establish decay rates for the derivatives of the Ricci curvature. In the sequel, we frequently utilize cutoff functions detecting the time decay of the Ricci curvature beyond a fixed time. More precisely, if $|\mathrm{Rm}[g(t)]|_{g(t)} \leq c_0$ on $M\times [0,\infty)$ and $|\ric[g(t)]|_{g(t)} \leq c_1t^{-\beta}$ for $t\geq T$, we define the function $\psi:[0,\infty) \r [0,\infty)$ by
\begin{align} \label{eq:curvature-time-cutoff}
    \psi(t)=
    \begin{cases}
        n(n-1)c_0 & t\leq T \\ \chi n(n-1)c_0 + (1-\chi)c_1t^{-\beta} & T \leq t \leq T+1 \\ c_1t^{-\beta} & t\geq T+1,
    \end{cases}
\end{align}
where $\chi$ is a non-negative, smooth cutoff function such that $\chi \equiv 1$ on $[0,T+\frac{1}{2}]$ and $\chi \equiv 0$ on $[T+1,\infty)$. By definition, $\psi$ is uniformly bounded and $|\ric[g(t)]|_{g(t)}\leq \psi(t)$ for all $t\geq 0$. Moreover, if we assume that $\beta > 1$, then $I(t):=\int^t_0 \psi(u)du$ is uniformly bounded. We shall use similar constructions in other maximum principle arguments without explicitly defining the appropriate time cutoff functions. 

Before establishing some decay for $|\nabla \ric|$, we demonstrate the utility of the cutoff function $\psi$ by proving the following simpler result pertaining to the volume evolving along an ARF flow, which will prove invaluable in its own right in the next section.

\begin{lemma} \label{volume-upper-bound}
    If $g(t) \in \mathcal{M}^2_\beta(g_{\mathrm{RF}})$ for some $\beta > 1$, then there exists a positive constant $C$ independent of $t$ such that $V(t):=\mathrm{Vol}[g(t)] \leq C$.
    \begin{proof}
        Set $\tau = S-t$ so that $\partial_\tau = -\partial_t$. After integrating the variational formula $\partial_\tau(dV)=\scal dV$ over $M$, Lemma \ref{arf-ricci-decay}
       tells us that $\partial_\tau V\leq \psi V$. Consequently, the maximum principle tells us that $V(\tau) \leq V(0)\exp\left(\int^\tau_0 \psi(u)du\right)=V(0)e^{I(\tau)}$, hence the desired result follows from the observation that $I(\tau)$ is bounded provided that $\beta>1$.
    \end{proof}
\end{lemma}

The following result provides us with an asymptotic estimate for some of the derivatives of the Ricci curvature. The number of times we can differentiate while preserving some time decay is dictated by how much regularity we assume the ARF flow to have.
\begin{lemma} \label{arf-ricci-derivatives-decay}
    If $g(t) \in \mathcal{M}^{m+2}_\beta(g_{\mathrm{RF}})$ and $m\leq \beta$, then $|\nabla^k \ric[g(t)]| = O(|t|^{k-\beta})$ for any $k\leq m$. 
    \begin{proof}
        Lemma \ref{arf-ricci-decay} tells us that the desired estimate holds true when $m=0$. Let us now suppose that the estimate holds for all $j\in \{0,...,m-1\}$ for some $m\geq 1$. In particular, if $g(t) \in \mathcal{M}^{m+2}_\beta(g_{\mathrm{RF}})$, then for every $j\leq m-1$, there exist constants $C_{j}>0$ and $T_j\gg 1$ such that $|\nabla^{j}\ric|^2 \leq C_{j}|t|^{2j-2\beta-2}$ for all $t\leq -T_j$. We define $T_m$ to be the maximum of the quantities $T_j$ for $j\leq m-1$ and show that there exists a constant $C_m$ such that $|\nabla^m \ric|^2 \leq C_m |t|^{2m-2\beta-2}$ for all $t< -T_m$.
          
        Let us fix an arbitrary negative time $s < -T_m$ and consider the subsolution defined on $[3s,0]$. Emulating our construction of the cutoff function \eqref{eq:curvature-time-cutoff}, we can whip up a bounded, non-negative cutoff function $\xi_s:[3s,0] \r [0,|2s|^{2\beta-2m+1}]$ vanishing on $[3s,5s/2]$ such that $\xi_s(t)=|t|^{2\beta - 2m+1}$ for all $t \in [2s,0]$. We may further assume that $|\partial_t\xi|\leq c|s|^{2\beta-2m}$ for some constant $c$ depending only on $\beta$ and $m$. We define the function $F_s:M\times [3s,0] \r [0,\infty)$ by the formula
        \begin{align} \label{eq:ricci-decay-test-function}
            F_s(x,t) = \xi_s(t)[|\nabla^m \ric|^2(x,t) +A|\nabla^{m-1}\ric|^2(x,t)],
        \end{align}
        where $A$ is a constant to be determined momentarily. Since $M$ is compact and $\xi_s$ is defined independently of $s$ on $[-T_m,0]$, $\square F_s$ is bounded from above on $M\times [-T_m,0]$ by a constant independent of $s$. 
        
        The evolution equation \eqref{eq:heat-eqn-ricci-derivates-norm-sqaured} permits us to estimate $\square |\nabla^j \ric|^2$ for $j \leq m$ and $t\in [2s,-T_m]$ as follows. We extract a quadratic $\nabla^j\ric$ term from the $i=0$ and $i=j$ summands. To tend to the remaining summands, we bound the derivatives of the Riemann curvature tensor from above by a uniform constant and estimate $|\nabla^{j-i}\ric|$ using the inductive hypothesis. In particular, we obtain linear $\nabla^j\ric$ terms with coefficients of order $O(|t|^{i-\beta-1})$, each of which can be bounded from above by a coefficient of order $O(|t|^{j-\beta-1})$. We deduce from these observations that
        \begin{align}
            \square|\nabla^j \ric|^2 \leq -2|\nabla^{j+1}\ric|^2 + B_{1,j}|\nabla^j \ric|^2 +B_{2,j}|t|^{j-\beta-1}|\nabla^j\ric|. 
        \end{align}
        We caution the reader that the constants $B_i$ may vary from line to line. Substituting $j=m-1$ and $j=m$ variants of this estimate into the formula for the action of the heat operator on $F$, we obtain that
        \begin{align}
            \square F_s\leq [&(B_1-2A)\xi + (\partial_t\xi)] \cdot |\nabla^m \ric|^2 + B_2\xi |t|^{m-\beta-1}|\nabla^m\ric| \\ &+[B_3A\xi + A(\partial_t\xi)] \cdot |\nabla^{m-1}\ric|^2 + B_4A\xi |t|^{m-\beta-2}|\nabla^{m-1}\ric|.
        \end{align}
        Recalling that $\xi = |t|^{2\beta-2m+1}$ on $[2s,-T_m]$, we can write this estimate explicitly as
        \begin{align}
            \square F_s &\leq [(B_1-2A)|t|^{2\beta -2m+1}-(2\beta-2m+1)|t|^{2\beta-2m}] \ |\nabla^m \ric|^2 + B_2|t|^{\beta - m}  |\nabla^m \ric|  \\ &\text{ } \text{ } \text{ } \text{ } + [B_3A|t|^{2\beta-2m+1} - A(2\beta-2m+1)|t|^{2\beta-2m}] |\nabla^{m-1}\ric|^2 + B_4A|t|^{\beta-m-1} |\nabla^{m-1}\ric|.
        \end{align}
        Since $m\leq \beta$, we can jettison some of the terms and apply Young's inequality twice to further simplify this upper bound to
        \begin{align}
            \square F_s \leq 2(B_1-2A) |t|^{2\beta-2m+1}  |\nabla^m\ric|^2 + 2AB_2 |t|^{2\beta-2m+1} |\nabla^{m-1}\ric|^2.
        \end{align}
        Consequently, if we set $A = B_1$, 
        then the first term becomes non-positive. Furthermore, by induction, $|\nabla^{m-1}\ric|^2 \leq C_{m-1}|t|^{2m-2\beta-2}$ for all $t\leq -T_{m}$, making the second term, and therefore $\square F_s$, uniformly bounded by a constant independent of $s$ on $M\times [2s,-T_m]$. Using the presumed uniform bounds for $\xi_s$ and $\partial_t\xi_s$, by enlarging $A$ if necessary, analogous computations allow one to draw the same conclusion on $M\times [3s,2s]$. The maximum principle therefore implies that
        \begin{align}
            \xi_s |\nabla^m \ric|^2 \leq F_s \leq C_m(t-3s)
        \end{align}
        for all $t\in [3s,0]$, where $C_m$ depends on $\max_{M\times [-T_m,0]}|\nabla^m \ric|^2$, the constants $C_j$ for each $j\leq m-1$, and the uniform bounds for the norms of the Riemann curvature and its derivatives. In particular, $C_m$ is independent of $s$.
        Since $\xi_s(s) = |s|^{2\beta-2m+1}$, it follows that $|\nabla^m \ric|^2(s) \leq 2C_m|s|^{2m-2\beta}$. Since the choice of $s< -T_m$ is arbitrary, the desired result follows.
    \end{proof}
\end{lemma}

\section{A dynamical energy functional on ancient ARF Ricci flows} \label{dyn-compact-ARF}
\subsection{Estimates for conjugate heat flows}\label{backward-heat-flow-estimates} Suppose that $(M^n,g(t))_{t\in (-\infty,T]}$ is an ancient solution to the Ricci flow equation on a compact manifold $M^n$. We also assume that $g(t) \in \mathcal{M}^3_\beta(g_{\mathrm{RF}})$ for some $\beta \geq 2$. Given a fixed time $s$, we denote by $f^s$ the solution to the conjugate heat flow
\begin{align} \label{eq:backward-heat-flow-t}
    \partial_tf^s = -\Delta f^s + |\nabla f^s|^2 - \scal
\end{align}
with initial datum $f^s(\cdot,s)=0$. While there is nothing special about this initial condition, it streamlines the proofs of the forthcoming estimates,
all of which apply to any solution to \eqref{eq:backward-heat-flow-t} with initial datum $f^s(\cdot,s)=u$ for any smooth function $u$ on $M$ with bounded derivatives with respect to $g(s)$ (or equivalently, any other metric along the flow).

We reparameterize time by setting $\tau(t)=s-t$ so that $\partial_\tau=-\partial_t$. With respect to $\tau$, the conjugate heat flow equation becomes
\begin{align} \label{eq:backward-heat-flow-tau}
    \partial_\tau f^s = \Delta f^s - |\nabla f^s|^2 + \scal,
\end{align}
which is an equation readily lending itself to maximum principle applications. We henceforth denote by $\square$ the heat operator $\square = \partial_\tau - \Delta$ with respect to $\tau$ and conflate $f^s$ with its reparameterization with respect to $\tau$.

Our goal in this subsection is to establish uniform upper bounds independent of $s$ for the norms of $f^s$, its gradient, and its Laplacian. Such bounds will prove paramount in our construction of a dynamical energy functional defined in terms of the aforementioned one-parameter family of conjugate heat flows. Let us begin with the following elementary upper bound for $|f^s|$. We caution the reader that we suppress the $s$-superscripts pervasively throughout the forthcoming proofs.

\begin{lemma} \label{function-arf-bound}
    There exists a positive dimensional constant $a(n)>0$ such that $0 \leq f^s \leq a(n)$ for all $s$.
    \begin{proof}
        Since ancient flows have non-negative scalar curvature, the non-negativity of $f^s$ may be deduced from the computations presented in the proof of \cite[Lemma 5.4]{lopez-ozuch}.
        
        Let us turn our attention to the stated upper bound. Since $|\scal|\leq \sqrt{n} \cdot |\ric|$,  \eqref{eq:backward-heat-flow-tau} implies that
        \begin{align}
            \square f^s(x,\tau) &= -|\nabla f^s|^2 + \scal \leq  \sqrt{n} \cdot \psi(\tau)
        \end{align}
        for any $x\in M$, where the function $\psi$ is defined as in \eqref{eq:curvature-time-cutoff}. The solution $\phi$ to the ODE $\partial_\tau \phi=\sqrt{n} \cdot \psi$ with initial condition $\phi(0)=0$ satisfies
        \begin{align}
            \phi(\tau) = \sqrt{n}\int^\tau_0 \psi(u)du \leq (T+1)\left(n^{\frac{3}{2}}(n-1)c_0+\frac{c_1\sqrt{n}}{(1-\beta)(T+1)^\beta}\right)=:a(n)
        \end{align}
        for any $\tau \geq 0$. Consequently, since $f^s(0)=0$, the maximum principle tells us that
        \begin{align}
            f^s(x,\tau) \leq \phi(\tau) \leq a(n)
        \end{align}
        for any $(x,\tau) \in M\times [0,\infty)$, as desired.
    \end{proof} 
\end{lemma}

In a similar vein, we use cutoff functions detecting the decay rates of the first derivatives of curvature terms to establish a uniform upper bound for $|\nabla f^s|$.

\begin{lemma} \label{gradient-compact-decay}
    There exists a positive dimensional constant $b(n)>0$ such that $|\nabla f^s| \leq b(n)$ for all $s$.
    \begin{proof}
        Using the identity $\partial_\tau(\nabla f)=\nabla(\partial_\tau f)+\nabla f*\ric$ and the fact that $\partial_\tau g = 2\ric$, we compute that
        \begin{align}
             \partial_\tau|\nabla f|^2 &= -2\ric(\nabla f,\nabla f)+2\langle \nabla(\partial_\tau f),\nabla f \rangle + (\nabla f)^{*2}*\ric.
        \end{align}
        Substituting \eqref{eq:backward-heat-flow-tau} and absorbing the $-2\ric(\nabla f,\nabla f)$ term into the $(\nabla f)^{*2}*\ric$ term, we simplify this evolution equation to
        \begin{align} \label{eq:gradient-compact-decay-bound-0}
            \partial_\tau|\nabla f|^2 = 2\langle \nabla(\Delta f - |\nabla f|^2+\scal),\nabla f \rangle + (\nabla f)^{*2}*\ric.
        \end{align}
        Furthermore, by the Bochner formula, we have that
        \begin{align} \label{eq:gradient-compact-decay-bound-0-bochner}
            \Delta|\nabla f|^2 = 2\langle \nabla \Delta f,\nabla f \rangle + (\nabla f)^{*2}*\ric +2|\nabla^2f|^2.
        \end{align}
        Subtracting \eqref{eq:gradient-compact-decay-bound-0-bochner} from  \eqref{eq:gradient-compact-decay-bound-0} yields that
        \begin{align} 
            \square |\nabla f|^2 = 2\langle \nabla \scal,\nabla f \rangle - 2|\nabla^2f|^2 + (\nabla f)^{*2}*\ric -2\langle \nabla|\nabla f|^2,\nabla f \rangle. \label{eq:gradient-compact-decay-bound-2}
        \end{align}
        By Lemmas ~\ref{arf-ricci-decay} and ~\ref{arf-ricci-derivatives-decay}, by emulating the construction of the function defined by \eqref{eq:curvature-time-cutoff}, we may construct cutoff functions $\psi_1,\psi_2:[0,\infty) \r [0,\infty)$ such that 
        \begin{align}
            |\ric|(x,\tau) \leq \psi_1(\tau) = O(\tau^{-\beta}) \quad \mathrm{and} \quad |\nabla \scal|(x,\tau) \leq \psi_2(\tau) = O(\tau^{1-\beta}).
        \end{align}
        Consequently, there is a positive constant $C_1$ such that
        \begin{align} \label{eq:gradient-compact-decay-bound-1}
            \square |\nabla f|^2 &\leq C_1[\psi_1(\tau)|\nabla f|^2+\psi_2(\tau)|\nabla f|] - 2\langle \nabla|\nabla f|^2,\nabla f \rangle.
        \end{align}
        Furthermore, an application of Young's inequality yields that
        \begin{align}
            2\psi_2(\tau)|\nabla f| = 2\psi_2(\tau)^{\frac{1}{2}}\psi_2(\tau)^{\frac{1}{2}}|\nabla f| \leq \psi_2(\tau) + \psi_2(\tau)|\nabla f|^2,
        \end{align}
        allowing us to absorb the $|\nabla f|$ term into the $|\nabla f|^2$ term in \eqref{eq:gradient-compact-decay-bound-1}. In particular, defining $\psi_3=\psi_1+\psi_2=O(\tau^{1-\beta})$, we obtain that
        \begin{align} \label{eq:gradient-compact-decay-bound-3}
            \square |\nabla f|^2 \leq \psi_3(\tau)(|\nabla f|^2+1)-2\langle \nabla|\nabla f|^2,\nabla f \rangle.
        \end{align}
        Since $\beta \geq 2$, the solution $\phi$ to the ODE $\partial_\tau \phi = \psi_3(\phi+1)$ with initial condition $\phi(0)=0$ satisfies
        \begin{align}
            \phi(\tau) &= \exp\left(\int^\tau_0 \psi_3(u)du \right)-1 \leq \exp\left(C_2(T+1)^{2-\beta}\right)=:b(n)^2
        \end{align}
        for all $\tau \geq 0$, where $C_2$ is a dimensional constant. Since $|\nabla f|(0)=0$, the maximum principle therefore implies that $|\nabla f|^2\leq b(n)^2$, as desired.

    \end{proof}
\end{lemma}

While we cannot guarantee a uniform upper bound for the norm of the Hessian of $f^s$, we can derive a uniform bound for its Dirichlet energy, which is sufficient for our principal application in the proceeding subsection and whose proof provides an appealing deviation from maximum principle arguments. It is also possible, albeit redundant and cumbersome, to obtain a uniform upper bound for the norm of the Laplacian of $f^s$.

 \begin{lemma} \label{laplacian-arf-compact-bound}
    There exists a positive dimensional constant $c(n)>0$ such that $\int_M |\nabla^2 f^s|^2 \leq c(n)$ for all $s$.
    \begin{proof}
        The gist of the proof is to integrate the evolution (in)equalities derived in Lemma ~\ref{gradient-compact-decay} and integrate by parts to tend to all the potentially unpleasant terms. We first establish a uniform bound for the $L^2$-norm of $|\Delta f^s|$ and then apply the Bochner formula to upgrade it to a uniform $L^2$-bound for $|\nabla^2f^s|$. It is crucial to emphasize that all of the constants mentioned in the proof are polynomials in the uniform bounds for the curvature, its derivatives, $|f^s|$, and $|\nabla f^s|$, and are therefore independent of $s$.
        
        Integrating \eqref{eq:gradient-compact-decay-bound-0} over $M$, we obtain that
        \begin{align}
            \int_M \partial_\tau|\nabla f|^2 dV_g &= 2\int_M \langle \nabla(\Delta f-|\nabla f|^2+\scal),\nabla f \rangle dV_g + \int_M [(\nabla f)^{*2}*\ric ]dV_g.
        \end{align}
        Since $M$ is compact, we may integrate by parts to write the previous equation as
        \begin{align} \label{eq:laplacian-compact-arf-bound-0}
            \int_M \partial_\tau|\nabla f|^2dV_g=-2\int_M |\Delta f|^2dV_g + \underbrace{\int_M [2\langle \nabla(\scal-|\nabla f|^2),\nabla f \rangle +(\nabla f)^{*2}*\ric]dV_g}_{:=J(\tau)}.
        \end{align}
        We avail of our previously established estimates for $\ric$, $\nabla \ric$, $f$, and $\nabla f$ to obtain a uniform upper bound for $|J(\tau)|$. In particular, as we were made privy to in the proof of Lemma \ref{gradient-compact-decay},
        $|\langle \nabla \scal,\nabla f \rangle|$ and $|(\nabla f)^{*2}*\ric|$ are uniformly bounded, so $\int_M [\langle \nabla \scal,\nabla f \rangle + (\nabla f)^{*2}*\ric] dV_g$ is as well by Lemma ~\ref{volume-upper-bound}. Consequently, there exists a positive constant $C_1$ such that
        \begin{align}
            |J(\tau)| \leq C_1 + 2\left|\int_M \langle \nabla|\nabla f|^2,\nabla f \rangle  dV_g \right|
        \end{align}
        This upper bound,  \eqref{eq:laplacian-compact-arf-bound-0} and the triangle inequality together imply that
        \begin{align}
            \left| \int_M \partial_\tau|\nabla f|^2 dV_g \right| &\geq 2\int_M |\Delta f|^2 dV_g - |J(\tau)| \\ &\geq 2\int_M |\Delta f|^2 dV_g - C_1 - 2\left|\int_M \langle \nabla|\nabla f|^2,\nabla f \rangle  dV_g \right|. \label{eq:laplacian-compact-arf-bound-6}
        \end{align}
        Furthermore, integrating the upper bound \eqref{eq:gradient-compact-decay-bound-3}, we find that
        \begin{align}
            \int_M \square |\nabla f|^2 dV_g \leq \psi_3(\tau)\int_M (|\nabla f|^2+1)dV_g - 2\int_M \langle \nabla|\nabla f|^2,\nabla f \rangle dV_g.
        \end{align}
        Since $\psi_3(\tau)=O(\tau^{1-\beta})$ and $|\nabla f|^2$ is uniformly bounded per Lemma \ref{gradient-compact-decay}, it follows that there is a positive constant $C_2$ such that
        \begin{align} \label{eq:laplacian-compact-arf-bound-ac}
            \left|\int_M \square |\nabla f|^2 dV_g\right| \leq C_2+2\left|\int_M \langle \nabla|\nabla f|^2,\nabla f \rangle dV_g\right|.
        \end{align}
        The compactness of $M$ allows us to recognize the left-hand side as $|\int_M \partial_\tau |\nabla f|^2dV_g|$. Consequently, by comparing the previous upper bound to \eqref{eq:laplacian-compact-arf-bound-6}, we obtain that
        \begin{align} \label{eq:laplacian-compact-arf-bound-1}
            2\int_M |\Delta f|^2 dV_g \leq C_3 + 4\left|\int_M \langle \nabla|\nabla f|^2,\nabla f \rangle dV_g \right|,
        \end{align}
        where $C_3=C_1+C_2$. We handle the integral on the right-hand side using the Cauchy-Schwarz inequality and integration by parts. In particular, we compute that
        \begin{align}
            \left|\int_M \langle \nabla|\nabla f|^2,\nabla f \rangle dV_g\right| = \left|\int_M (|\nabla f|^2 \cdot \Delta f)dV_g\right| \leq \left(\int_M |\Delta f|^2 dV_g\right)^{\frac{1}{2}}\left(\int_M |\nabla f|^4 dV_g\right)^{\frac{1}{2}}.
        \end{align}
        Combining this upper bound with \eqref{eq:laplacian-compact-arf-bound-1}, we obtain that
        \begin{align}
            \int_M |\Delta f|^2 dV_g - 2\left(\int_M |\Delta f|^2 dV_g\right)^{\frac{1}{2}}\left(\int_M |\nabla f|^4 dV_g\right)^{\frac{1}{2}} \leq \frac{C_3}{2}. 
        \end{align}
        Moreover, since $|\nabla f|^4$ is uniformly bounded, so is $\int_M |\nabla f|^4 dV_g$ per Lemma \ref{volume-upper-bound}, so some algebra provides us with a positive constant $C_4$ arising from a uniform bound for $\int_M |\nabla f|^2dV_g$ such that
        \begin{align}
           \left|\left(\int_M |\Delta f|^2 dV_g\right)^{\frac{1}{2}} - \left(\int_M |\nabla f|^4 dV_g\right)^{\frac{1}{2}}\right| \leq C_4.
        \end{align}
        At long last, the triangle inequality yields the uniform upper bound
        \begin{align}
            \left(\int_M |\Delta f|^2 dV_g\right)^{\frac{1}{2}} &\leq \left|\left(\int_M |\Delta f|^2 dV_g\right)^{\frac{1}{2}} - \left(\int_M |\nabla f|^4 dV_g\right)^{\frac{1}{2}}\right| + \left(\int_M |\nabla f|^4 dV_g\right)^{\frac{1}{2}} \leq C_5.
        \end{align}
        To obtain the Hessian bound, we apply the Bochner formula and integrate by parts to obtain that
        \begin{align}
            \int_M |\nabla^2f|^2 dV_g &= \frac{1}{2}\int_M[\Delta |\nabla f|^2 -2\langle \nabla \Delta f,\nabla f \rangle- 2\ric(\nabla f,\nabla f)]dV_g  \\ &= \int_M |\Delta f|^2 dV_g - \int_M \ric(\nabla f,\nabla f)dV_g. \label{eq:hessian-compact-bound}
        \end{align}
        Since $|(\nabla f)^{*2}*\ric|$ is uniformly bounded, so is $\int_M \ric(\nabla f,\nabla f)dV_g$, and $\int_M |\Delta f|^2dV_g$ too is uniformly bounded by virtue of our uniform bound for $\int_M |\Delta f|^2 dV_g$. Consequently, \eqref{eq:hessian-compact-bound} implies that $\int_M |\nabla^2 f|^2dV_g$ is uniformly bounded as well.
    \end{proof}
    \end{lemma}

\subsection{Properties of the dynamical energy functional for ARF flows} With the conjugate heat flow estimates out of the way, we are in a good position to define an analogue of \cite[Definition 5.3]{lopez-ozuch} for ARF flows.

If $g(t) \in \mathcal{M}^3_\beta(g_{\mathrm{RF}})$ is an ARF flow such that $\beta\geq 2$ and $s\leq 0$, we can define the pointed functional $\lambda^s_{\mathrm{dyn}}:(-\infty,s] \r \R$ by
    \begin{align}
        \lambda^s_{\mathrm{dyn}}(t) = \mathcal{F}[g(t),f^s(t)],
    \end{align}
    where $f^s$ is the solution to the backward heat flow \eqref{eq:backward-heat-flow-t} with initial condition $f^s(\cdot,s)\equiv 0$. By \cite[Section 5.4]{chow-ricci-flow-1}, this functional satisfies the monotonicity formula
\begin{align} \label{eq:conj-heat-flow-classical-variation}
    \frac{d}{dt} \lambda^s_{\mathrm{dyn}}(t) = 2\int_M |\ric_{f^s}|^2 e^{-f^s} dV_g \geq 0.
\end{align}
The dynamical $\lambda$-functional for ALE manifolds introduced in \cite[Section 5]{lopez-ozuch} is the limit infimum of $\lambda^s_{\mathrm{dyn}}(t)$ as $s$ tends to $-\infty$. In the setting of compact ARF flows in $\mathcal{M}^3_\beta(g_{\mathrm{RF}})$ for some $\beta \geq 2$, we construct a functional $\lambda^\infty_{\mathrm{dyn}}$ independent of $s$ that provides an upper bound for the usual $\lambda$-functional at each time $t$. The construction hinges on the existence of a function $f^\infty$ that can be realized as a limit of a subsequence of the functions $f^s$. To ensure that each function $f^s$ is defined on the same domain, we continue to set $\tau(t)=s-t$ so that the reparameterization of $f^s$ with respect to $\tau$ is defined on $M^n \times [0,\infty)$. As usual, we conflate $f^s$ with its corresponding reparameterization.

Observe that the equation \eqref{eq:backward-heat-flow-tau} is nonlinear, so parabolic theory does not readily imply that a weak $W^{2,2}(M\times [0,\infty))$-limit of a subsequence of the functions $f^s$ can be upgraded to a $\C^{1,\alpha}(M\times [0,\infty))$-limit. We amend the argument by instead considering the linear PDE satisfied by the functions $e^{-f^s}$ and transmuting the limit obtained from these functions into a solution to \eqref{eq:backward-heat-flow-tau}. Peruse \cite[Section 1]{deruelle-ozuch-2020} for a description of an analogous procedure used to obtain minimizers for Perelman's energy on ALE manifolds.

\begin{lemma} \label{limiting-function-lemma}
    There exists a subsequence $u^{s_i}$ of the functions $\exp(-f^s)$ converging in $\C^{1,\alpha}(M \times [0,\infty))$ to a function $u^\infty(\tau)$, and the function $f^\infty:M^n \times [0,\infty) \r \R$ defined by $f^\infty=-\log u^\infty $ is a solution to the (reparameterized) conjugate heat flow \eqref{eq:backward-heat-flow-tau}.
    \begin{proof}
        Define for each $s$ the function $u^{s}=\exp(-f^s)$, which satisfies the heat-type equation
        \begin{align} \label{eq:heat-exp-of-heat-flow}
            \partial_\tau u^{s} = \Delta u^{s} - \scal u^{s}.
        \end{align}
        Let us now fix a compact time interval $[0,a]$. 
        Per Lemmas \ref{function-arf-bound}, \ref{gradient-compact-decay}, and \ref{laplacian-arf-compact-bound}, the sequence $u^s$ is bounded in $W^{2,2}(M \times [0,a])$, hence a subsequence of $u^{s_i}$ (which we still denote by $u^{s_i}$) converges weakly in $W^{2,2}(M \times [0,a])$ to a function $u^\infty$, and parabolic theory tells us that $u^{s_i} \r u^\infty$ in $\C^{1,\alpha}(M \times [0,a])$ for some $\alpha>0$ and
        \begin{align} \label{eq:u-infty-parabolic-PDE}
            \partial_\tau u^\infty = \Delta u^\infty - \scal u^\infty
        \end{align}
        on $M\times [0,a]$. Furthermore, by Lemma \ref{function-arf-bound}, $f^{s_i}$ is non-negative and uniformly bounded in $s_i$, there exists a positive constant $c$ independent of $i$ such that $u^{s_i} \ge e^{-c}$ for each $i$, hence
        \begin{align}
            e^{-c} \leq u^\infty \leq 1
        \end{align}
        on $M\times [0,a]$. Consequently, $u^\infty$ can be extended to a solution to \eqref{eq:u-infty-parabolic-PDE} that is defined on a time interval strictly containing $[0,a]$, so by reiterating the previous argument, we can extend $u^\infty$ to a solution to \eqref{eq:u-infty-parabolic-PDE} that is defined on $M\times [0,\infty)$. Consequently, the function $f^\infty:=-\log(u^\infty)$ is well-defined and satisfies
        \begin{align}
           \partial_\tau f^\infty = \Delta f^\infty - |\nabla f^\infty|^2 + \scal
        \end{align}
        on $M \times [0,\infty)$.
    \end{proof}
\end{lemma}

We now define the canonical dynamical functional $\lambda^\infty_{\mathrm{dyn}}$ in terms of the limiting function $f^\infty$ obtained in the previous lemma.

\begin{definition}[A dynamical functional for ARF flows]
    If $g(t) \in \mathcal{M}^3_\beta(g_{\mathrm{RF}})$ is an ARF flow with bounded curvature and $\beta\geq 2$, we define the functional $\lambda^\infty_{\mathrm{dyn}}:(-\infty,0] \r \R$ by
    \begin{align}
        \lambda^\infty_{\mathrm{dyn}}(t) = \mathcal{F}[g(t),f^\infty(-t)].
    \end{align}
\end{definition}

As promised, we have the following relationship between our newly introduced dynamical $\lambda$-functional and the ordinary $\lambda$-functional.

\begin{proposition}
     The dynamical $\lambda$-functional satisfies the inequality $\lambda^\infty_{\mathrm{dyn}}(t) \geq \lambda(t)$.
     \begin{proof}
         After parabolically rescaling the ARF solution, we may assume that $\mathrm{Vol}[g(0)]=1$. We set $f=f^\infty$ and introduce the time reparameterization $\tau(t)=-t$. Since $\partial_\tau(dV_g)=\scal \cdot dV_g$ and $\partial_\tau f = \Delta f -|\nabla f|^2+\scal$, the first variation of the weighted volume with respect to $f$ is given by
         \begin{align}
             \partial_\tau \int_M e^{-f}dV_g &= \int_M (|\nabla f|^2-\Delta f)e^{-f}dV_g = 0.
         \end{align}
         Consequently, $\int_M e^{-f} dV_g$ is constant in time, so $\int_M e^{-f(t)} dV_{g(t)}= \int_M e^{-f(0)} dV_{g(0)}=1$ for any $t$. In particular, $f(\cdot,t)$ is a valid test function for any $t\leq 0$, so the desired inequality follows from the definition of the $\lambda$-functional.
     \end{proof}
\end{proposition}

Considering how $t \mapsto f^\infty(-t)$ is a solution to the original conjugate heat flow \eqref{eq:backward-heat-flow-t}, the classical formula \eqref{eq:conj-heat-flow-classical-variation} for the first variation of Perelman's energy functional along a conjugate heat flow coupled to a Ricci flow holds true for $\lambda^\infty_{\mathrm{dyn}}$.

\begin{proposition}[First variation of $\lambda^\infty_{\mathrm{dyn}}$]
    The first variation of $\lambda^\infty_{\mathrm{dyn}}$ is given by the formula
    \begin{align} \label{eq:lambda-infty-first-variation}
        \frac{d}{dt}\lambda^\infty_{\mathrm{dyn}}(t) = 2\int_M |\ric_{f^\infty}|^2 e^{-f^\infty} dV_g.
    \end{align}
\end{proposition}

A notable property of $\lambda^\infty_{\mathrm{dyn}}$ is that its critical points are precisely those coupled solutions that have constant conjugate heat flows.
\begin{corollary}
     The functional $\lambda^\infty_{\mathrm{dyn}}$ is constant if and only if $f^\infty$ is constant in space and time. In particular, 
     \begin{enumerate}[(a)]
         \item if there exist two distinct times $t_1,t_2 \in (-\infty,0]$ such that $t_1<t_2$ and $\lambda^\infty_{\mathrm{dyn}}(t_1)=\lambda^\infty_{\mathrm{dyn}}(t_2)$, then $(g(t))_{t\in [t_1,t_2]}$ is a Ricci-flat, steady gradient Ricci soliton;
         \item if there exists a decreasing sequence of times $t_n \in (-\infty,0]$ tending to $-\infty$ such that $\lambda^\infty_{\mathrm{dyn}}(t_n)=\lambda^\infty_{\mathrm{dyn}}(t_{n+1})$ for every $n\in \mathbb{N}$, then $g(t)=g_{\mathrm{RF}}$ on $(-\infty,t_1]$.
     \end{enumerate}

    \begin{proof}
        First suppose that $\lambda^\infty_{\mathrm{dyn}}$ is constant. Then the formula \eqref{eq:lambda-infty-first-variation} implies that $\ric = -\nabla^2f^\infty$. Tracing this equation yields that $\scal = -\Delta f^\infty$, whence
        \begin{align} \label{eq:f-infty-first-variation-lambda-constant-1}
            \partial_tf^\infty = -\Delta f^\infty + |\nabla f^\infty|^2 - \scal = |\nabla f^\infty|^2.
        \end{align}
        Furthermore, we integrate by parts to compute that
        \begin{align}
            \lambda^\infty_{\mathrm{dyn}}(t) = \int_M (|\nabla f^\infty|^2 + \scal)e^{-f^\infty}dV_g=\int_M (|\nabla f^\infty|^2 - \Delta f^\infty)e^{-f^\infty}dV_g =0.
        \end{align}
        Since $\scal \geq 0$, it follows that $\scal=0$ and $\nabla f^\infty \equiv 0$, making $f^\infty$ constant in space. Since $\partial_tf^\infty = |\nabla f^\infty|^2$, the constancy of $f^\infty$ in space implies that it is also constant in time.

        Conversely, suppose that $f^\infty$ is constant in space and time. Then $\partial_tf^\infty=\Delta f^\infty = \nabla f^\infty = 0$. Since $f^\infty$ satisfies the first equality of \eqref{eq:f-infty-first-variation-lambda-constant-1}, it follows that $\scal \equiv 0$. Consequently, $\lambda^\infty_{\mathrm{dyn}}\equiv 0$.

        Turning our attention to part (a), since $\lambda^\infty_{\mathrm{dyn}}$ is non-decreasing, the previous argument tells us that $f^\infty$ is constant on $[t_1,t_2]$, hence $\ric[g(t)]=-\nabla^{2,g(t)}f=0$ for all $t\in [t_1,t_2]$. As for part (b), select any time $t\leq t_1$. Since $t\in [t_{n_0+1},t_{n_0}]$ for some $n_0\in \mathbb{N}$, part (a) implies that $g(t)=g(t_n)$ for all $n\in \mathbb{N}$. It then follows from the ARF assumption that  $$|g(t)-g_{\mathrm{RF}}| = |g(t_n)-g_{\mathrm{RF}}| \leq C_0|t_n|^{-\beta}$$
        for every $n\in \mathbb{N}$ for some constant $C_0$ independent of $n$, readily implying the result.
    \end{proof}
\end{corollary}

\section{Eigenvalue estimates for conjugate heat flows} \label{eigenvalue-bounds}
In \cite{colding-minicozzi}, Colding and Minicozzi derive sharp upper bounds for eigenvalues of the drift Laplacian for a modified Ricci flow. Staying on the topic of backward heat flows, we dedicate this section to proving analogues of some of their results for eigenvalues of the drift Laplacian for a modified backward heat flow. 

\subsection{Local upper bounds for the eigenvalues of conjugate heat flows}
Suppose we have a coupled system $(g(t),f(t))$ on a compact manifold $M^n$ satisfying 
\begin{align} \label{eq:eigenvalue-coupled-system}
    \begin{cases}
        \partial_tg = -2\ric_f \\ \partial_tf = - \scal - \Delta f,
    \end{cases}
\end{align}
Observe that the weighted volume form $e^{-f}dV_g$ is constant along this coupled flow, hence why the variational formula for $g$ is referred to as the normalized Ricci flow. Furthermore, this system is, up to diffeomorphism, equivalent to a coupling of the Ricci flow with a conjugate heat flow, which was the system we studied in the previous sections.

We recall that the drift Laplacian $\Delta_f=\Delta - \nabla_{\nabla f}$ has discrete spectrum with $\lambda_k \r \infty$ as $k\r \infty$, although the eigenvalues can have multiplicity. We derive the following local upper bound for the $k$th eigenvalue of $\Delta_f$ along \eqref{eq:eigenvalue-coupled-system}.

\begin{proposition}
    If $(g(t),f(t))$ is a coupled system satisfying \eqref{eq:eigenvalue-coupled-system} and $\lambda_k(t)$ denotes the $k$th eigenvalue of $(M^n,g(t))$ for any $k \geq 1$, then
    \begin{align} \label{eq:eigenvalue-bound}
        \lambda_k(t) \leq \frac{\lambda_k(t_0)}{1-2(t-t_0)\lambda_k(t_0)} \quad \text{for} \text{ } \text{all} \text{ }  t \in \left[t_0,t_0+\frac{1}{2\lambda_k(t_0)}\right).
    \end{align}
    \begin{proof}
        Emulating the computations in \cite[Section 2]{colding-minicozzi} using the variational formulas provided in Lemma ~\ref{eigenvalue-evolution-1} and Corollary ~\ref{eigenvalue-evolution-2}, we obtain that $\lambda_k'(t) \leq 2\lambda_k^2(t)$, where the derivative $\lambda_k'$ is understood in the sense of \cite[(1.1)]{colding-minicozzi}. Consequently, $(-\frac{1}{2\lambda_k})'\leq 1$, hence
        \begin{align}
            \frac{1}{2\lambda_k(t_0)} - \frac{1}{2\lambda_k(t)} \leq t-t_0
        \end{align}
        for any $t<t_0+\frac{1}{2\lambda_k(t_0)}$. Algebraically manipulating this inequality yields \eqref{eq:eigenvalue-bound}.
    \end{proof}
\end{proposition}

\subsection{Sharpness}
We dedicate this subsection to establishing the rigidity of the eigenvalue upper bound \eqref{eq:eigenvalue-bound} for $k=1$. We define the function $f$ and the metric $g$ on $\R^n$ by
\begin{align}
    f(x,t) = \frac{|x|^2}{4} + \frac{n}{2}\log\left(\frac{u(t)}{u(t_0)}\right) \quad \mathrm{and} \quad g(t)=u(t)\delta_{ij},
\end{align}
where $u(t)$ is defined on $[t_0,t_0+u(t_0))$ by
\begin{align}
    u(t) = u(t_0) - (t-t_0).
\end{align}
One computes that
\begin{align}
    (\partial_tf,\partial_tg) = (-\Delta f,-2\nabla^2f),
\end{align}
which is readily seen to satisfy \eqref{eq:eigenvalue-coupled-system} since $\ric[g(t)]=0$. By \cite[Lemma 2.41]{colding-minicozzi}, the first eigenvalue of $g(t)$ is $\frac{1}{2u(t)}$, so $\lambda_1(t_0)=\frac{1}{2u(t_0)}$. Consequently,
\begin{align}
    \lambda_1(t) = \frac{1}{2u(t)} = \frac{1}{2[u(t_0)-(t-t_0)]} = \frac{\lambda_1(t_0)}{1-2(t-t_0)\lambda_1(t_0)},
\end{align}
thereby establishing the rigidity of the local estimate \eqref{eq:eigenvalue-bound} for the first eigenvalue.

\appendix
\section{Evolution equations and inequalities} \label{appendix}
Here, we record several computations that are utilized throughout the paper.
\subsection{First variations of norms of some curvature tensors} There are ample classical estimates that provide inductive local upper bounds for the first variations of the derivatives of the Riemann curvature tensor. Peruse \cite[Section 6.1]{chow-ricci-flow-1} for a succinct treatment of such estimates. Nevertheless, as is evident from the proofs in \S\ref{curvature-estimates} and \S\ref{dyn-compact-ARF}, such upper bounds our insufficient for our purposes considering how we require some modest \emph{global} time decay for the derivatives of the Ricci curvature tensor. Fortunately, when one is merely concerned about Ricci curvature estimates, it is possible to tweak the classical variational formulas enough to replace several Riemann curvature quantities with Ricci curvature quantities. 

\begin{lemma} \label{norm-squared-of-ricci-ac}
    Let $(M^n,g(t))_{t\in I}$ be a solution to the Ricci flow equation. Then
    \begin{align} \label{eq:evolution-of-norm-squared-of-ricci}
        \frac{\partial}{\partial t}|\ric|^2 \leq \Delta|\ric|^2 -2|\nabla \ric|^2+ c(n)|\mathrm{Rm}|\cdot |\ric|^2 
    \end{align}
    for some $c(n)>0$. 
    \begin{proof}
        We claim that
        \begin{align} \label{eq:ricci-ac-max-principle-bound-1}
            \partial_t(g^{ri}g^{sj})R_{rs}R_{ij} = 2g^{ri}g^{sj}R_{rs}(R_i^aR_{aj}+R_j^aR_{ia}).
        \end{align}
        To derive this formula, we first use geodesic normal coordinates in tandem with the Ricci flow equation to compute that
        \begin{align}
            \partial_t(g^{ri})g^{sj}R_{rs}R_{ij} = 2R_i^r g^{sj}R_{rs}R_{ij} = 2R_i^a R_{aj}R_{ij},
        \end{align}
        which coincides with
        \begin{align}
            2g^{ri}g^{sj}R_{rs}R_i^a R_{aj} = 2R_{ij}R_i^aR_{aj}.
        \end{align}
        An analogous computation for the second term arising from the product rule tells us that $g^{ri}(\partial_tg^{sj})R_{rs}R_{ij}=2g^{ri}g^{sj}R_{rs}R_j^aR_{ia}$, proving \eqref{eq:ricci-ac-max-principle-bound-1}. Similarly, to compute $g^{ri}g^{sj}\partial_t(R_{rs}R_{ij})$, we use the evolution equation for the Ricci tensor, namely
        \begin{align}
           \partial_tR_{ij} = \Delta R_{ij} + 2R_{kij\ell}R_{k\ell} - 2R_{ik}R_{jk},
        \end{align}
        to obtain that
        \begin{align} 
            g^{ri}g^{sj}\partial_t(R_{rs}R_{ij}) &= 2g^{ri}g^{sj}R_{rs}(\partial_tR_{ij}) = 2g^{ri}g^{sj}R_{rs}(\Delta R_{ij}+2R_{kij\ell}R_{k\ell} - 2R_{ik}R_{jk}). \\ &\label{eq:ricci-ac-max-principle-bound-2}
        \end{align}
        Furthermore, we observe that
        \begin{align}
            \Delta |\ric|^2 = 2\langle \Delta \ric,\ric \rangle +2|\nabla \ric|^2 = 2g^{ri}g^{sj}R_{rs}\Delta R_{ij} + 2|\nabla \ric|^2,
        \end{align}
        so we may rewrite \eqref{eq:ricci-ac-max-principle-bound-2} as
        \begin{align}
            g^{ri}g^{sj}\partial_t(R_{rs}R_{ij})= \Delta |\ric|^2 - 2|\nabla \ric|^2 + 2g^{ri}g^{sj}R_{rs}(2R_{kij\ell}R_{k\ell} - 2R_{ik}R_{jk}).
        \end{align}
        Adding this equation to \eqref{eq:ricci-ac-max-principle-bound-1}, we find that
        \begin{align}
            \partial_t|\ric|^2 &= \partial_t(g^{ri}g^{sj}R_{rs}R_{ij}) = \Delta |\ric|^2 - 2|\nabla \ric|^2 + B,
        \end{align}
        where $B=B_{aijk\ell rs}$ is the tensor
        \begin{align}
            2g^{ri}g^{sj}R_{rs}(R_{i}^a R_{aj} + R_j^a R_{ia} + 2R_{kij\ell}R_{k\ell} - 2R_{ik}R_{jk}) = \ric^{*3} + \mathrm{Rm}*(\ric)^{*2}.
        \end{align}
        Consequently, $|B| \leq c(n)|\mathrm{Rm}| \cdot |\ric|^2$ since $|\ric| \leq \sqrt{n(n-1)}|\mathrm{Rm}|$.
    \end{proof}
\end{lemma}

\begin{lemma} \label{norm-squared-of-first-derivative-ricci-ac}
    Let $(g(t))_{t\in I}$ be a solution to the Ricci flow equation. Then
    \begin{align} \label{eq:first-variation-of-nabla-Ric}
        \frac{\partial}{\partial t}|\nabla \ric|^2 &= \Delta |\nabla \ric|^2 - 2|\nabla \ric|^2 + \mathrm{Rm}*(\nabla \ric)^{*2} + (\nabla \mathrm{Rm})*\ric * (\nabla \ric).
    \end{align}
    \begin{proof}
        Since $\partial_tg=-2\ric$, we have that
        \begin{align}
            \partial_t(\nabla \ric) = \nabla(\partial_t\ric) + \nabla(\partial_tg)*\ric &= \nabla(\Delta \ric + \mathrm{Rm}* \ric + (\ric)^{*2}) + (\nabla \ric*\ric) \\ &= \nabla \Delta \ric + (\nabla \mathrm{Rm})*\ric + (\nabla \ric)*\mathrm{Rm}.
        \end{align}
        Furthermore, the Bochner formula tells us that
        \begin{align}
            2\langle \nabla \Delta \ric, \nabla \ric \rangle &= \Delta |\nabla \ric|^2 -2|\nabla \ric|^2 + \ric * (\nabla \ric)^{*2},
        \end{align}
        and consequently,
        \begin{align}
            \partial_t|\nabla \ric|^2 &= 2\langle \partial_t(\nabla \ric),\nabla \ric) + \ric*(\nabla \ric)^{*2} \\ &= \Delta |\nabla \ric|^2 - 2|\nabla \ric|^2 + \mathrm{Rm}*(\nabla \ric)^{*2} + (\nabla \mathrm{Rm})*\ric *(\nabla \ric). \label{eq:gradient-ricci-ac-decay-1}
        \end{align}
    \end{proof}
\end{lemma}

\begin{lemma}
 Let $(g(t))_{t\in I}$ be a solution to the Ricci flow equation. Then
    \begin{align} \label{eq:heat-eqn-ricci-higher-order-derivatives}
        \square(\nabla^m\ric) = \sum_{i=0}^m \left[\nabla^i\mathrm{Rm} * \nabla^{m-i}\ric\right],
    \end{align}
    and consequently,
    \begin{align} \label{eq:heat-eqn-ricci-derivates-norm-sqaured}
        \square|\nabla^m\ric|^2= -2|\nabla^{m+1}\ric|^2 + \sum_{i=0}^m [\nabla^i \mathrm{Rm} * \nabla^{m-i}\ric*\nabla^m\ric].
    \end{align}
    \begin{proof}
        The proof of the previous lemma tells us that \eqref{eq:heat-eqn-ricci-higher-order-derivatives} holds when $m=1$. Let us now suppose that  \eqref{eq:heat-eqn-ricci-higher-order-derivatives} holds up to some $m\in \mathbb{N}$. Then
        \begin{align}
            \partial_t(\nabla^{m+1}\ric) &= \nabla \partial_t(\nabla^m\ric) + \nabla^m \ric * \nabla \ric \\ &= \nabla \left[\Delta(\nabla^m\ric)+\sum_{i=0}^m \left[\nabla^i\mathrm{Rm}*\nabla^{m-i}\ric\right]\right] + \nabla^m\ric*\nabla\ric. \label{eq:ricci-derivative-evolution-3}
        \end{align}
        Using the standard commutation formula for exchanging a covariant derivative and a Laplacian, we obtain that
        \begin{align} \label{eq:ricci-derivative-evolution-1}
            \nabla \Delta(\nabla^m \ric) &= \Delta(\nabla^{m+1}\ric) + \mathrm{Rm}*\nabla^{m+1}\ric + \nabla \mathrm{Rm}*\nabla^m \ric.
        \end{align}
        Furthermore, since
        \begin{align}
            \nabla(\nabla^i \Rm * \nabla^{m-i}\ric) = \nabla^{i+1}\Rm * \nabla^{(m+1)-(i+1)}\ric + \nabla^i \mathrm{Rm} * \nabla^{(m+1)-i}\ric,
        \end{align}
        it follows that
        \begin{align} \label{eq:ricci-derivative-evolution-2}
            \nabla\left(\sum_{i=0}^m \left[\nabla^i\mathrm{Rm}*\nabla^{m-i}\ric\right]\right) = \sum_{i=0}^{m+1} \nabla^i \Rm * \nabla^{(m+1)-i}\ric.
        \end{align}
        Substituting \eqref{eq:ricci-derivative-evolution-1} and \eqref{eq:ricci-derivative-evolution-2} into \eqref{eq:ricci-derivative-evolution-3}, we obtain that
        \begin{align}
            \square(\nabla^{m+1}\ric) &= \sum_{i=0}^{m+1} \nabla^i \Rm * \nabla^{(m+1)-i}\ric \\ &+ \Rm *\nabla^{m+1} \ric + \nabla \Rm * \nabla^m \ric +\nabla^m \ric * \nabla \ric.
        \end{align}
        Observing that the final three terms may be absorbed into the first completes the inductive step and in turn the proof of the lemma.
    \end{proof}
\end{lemma}

\subsection{Evolution equations for quantities evolving along a modified backward heat flow}

\begin{lemma} \label{eigenvalue-evolution-1}
    Suppose we have a coupled system $(g(t),f(t))$ of Riemannian metrics and smooth functions on a compact manifold $M^n$ satisfying
    \begin{align} \label{eq:backward-modified-heat-flow-appendix}
        \begin{cases}
            \partial_tg = -2\ric - \hess_f \\ \partial_tf = -\scal - \Delta f.
        \end{cases}
    \end{align}  
    If $\partial_tu=\Delta_fu$ and $\partial_tv = \Delta_fv$, then the weighted $L^2$-inner product of $u$ and $v$ and the weighted $L^2$-norm of $|\nabla u|$ satisfy
    \begin{align}
        \partial_t\langle u,v \rangle_f = -2\int_M \langle \nabla u,\nabla v \rangle e^{-f} dV_g \quad \mathrm{and} \quad \partial_t||\nabla u||^2_{L^2(e^{-f}dV_g)} = -2\int_M |\nabla^2u|^2 e^{-f}dV_g.
    \end{align}
    \begin{proof}
        Using the product rule, we compute that
        \begin{align}
            \partial_t(uv)=(\partial_tu)v + u(\partial_tv) = v\Delta_fu + u\Delta_fv = \Delta_f(uv) - 2\langle \nabla u,\nabla v \rangle.
        \end{align}
        Since $\Delta_f(uv)e^{-f}$ integrates to zero and $\partial_t(e^{-f}dV_g)=0$, multiplying this series of equations by $e^{-f}$ and integrating over $M$ gives the first formula. As for the second formula, the chain rule and the weighted Bochner formula imply that
        \begin{align}
            \partial_t|\nabla u|^2 &= 2\langle \nabla(\partial_tu),\nabla u \rangle - (\partial_tg)(\nabla u,\nabla u \rangle \\ &= 2\langle \nabla(\Delta_fu),\nabla u \rangle + 2\ric_f(\nabla u,\nabla u ) \\ &= \Delta_f|\nabla u|^2 - 2|\nabla^2u|^2,
        \end{align}
        so multiplying this series of equations by $e^{-f}$ and integrating over $M$ yields the second formula.
    \end{proof}
\end{lemma}

\begin{corollary} \label{eigenvalue-evolution-2}
    Consider the quantities
\begin{align}
    J_{uv}(t) = \int (uv)e^{-f}, \quad D_{uv}(t) = \int \langle \nabla u,\nabla v \rangle e^{-f}, \quad I_u(t) = J_{uu}(t), \quad E_u(t)=D_{uu}(t).
\end{align}
If $\partial_tu = \Delta_fu$ and $\partial_tv = \Delta_fv$, then along \eqref{eq:backward-modified-heat-flow-appendix}, we have that
\begin{align}
    &J_{uv}'(t) = -2D_{uv}(t), \quad I_u'(t) = -2E_u(t), \quad E_u'(t) = -2\int_M |\nabla^2u|^2 e^{-f} \leq 0.
\end{align}
\end{corollary}

\bibliographystyle{alpha}
\bibliography{refs}

\end{document}